\documentclass[11pt]{amsart}
\input epsf
\usepackage{latexsym}
\usepackage[usenames,dvipsnames]{color}
\usepackage[colorlinks=true,linkcolor=blue,citecolor=BrickRed,urlcolor=RoyalBlue]{hyperref}
            
\usepackage{amssymb,amsmath,amsthm,amscd, amssymb,
graphicx, enumerate, verbatim, mathrsfs, color, url, bm, hyperref, parskip}
\usepackage[all]{xy}\usepackage[labelformat=empty]{caption}

\usepackage[left]{lineno}

\newtheorem{lem}[equation]{Lemma}

\newtheorem{thm}[equation]{Theorem}

\newtheorem{Example}[equation]{Example}

\newtheorem{remark}[equation]{Remark}

\def\co{\colon\thinspace}

\newcommand{\diam}{\mbox{diam}}

\newcommand{\e}{\varepsilon}

\def\G{\Gamma}

\def\b{\beta}

\def\s{\sigma}

\def\S1{\bf S^1}

\makeatletter
\def\equalsfill{$\m@th\mathord=\mkern-7mu
\cleaders\hbox{$\!\mathord=\!$}\hfill
\mkern-7mu\mathord=$}
\makeatother

\begin{document}

\abovedisplayskip=6pt plus3pt minus3pt
\belowdisplayskip=6pt plus3pt minus3pt

\title[Iterated circle bundles and infranilmanifolds]
{\bf Iterated circle bundles and infranilmanifolds}

\thanks{This work was partially supported by the Simons Foundation grant 524838.}

\author{Igor Belegradek}
\address{Igor Belegradek\\ School of Mathematics\\ Georgia Tech\\ Atlanta, GA, USA 30332\vspace*{-0.05in}}
\email{ib@math.gatech.edu}

\begin{abstract}
We give short proofs of the following two facts:
Iterated principal circle bundles are precisely the nilmanifolds.
Every iterated circle bundle is almost flat, and hence
diffeomorphic to an infranilmanifold.

\end{abstract}
\maketitle
\thispagestyle{empty}

A {\em infranilmanifold\,} is a closed manifold diffeomorphic to
the quotient space $N/\Gamma$ of a simply-connected nilpotent Lie group $N$ by
a discrete torsion-free subgroup $\Gamma$ of the semidirect product $N\rtimes C$
where $C$ is a maximal compact subgroup of $\mathrm{Aut}(N)$.
If $\G$ lies in the $N$ factor, the infranilmanifold is called a {\em nilmanifold.} 

An {\em iterated circle bundle\,} 
is defined inductively as the total space of a circle bundle 
whose base is an iterated circle bundle of one dimension lower,
and the base at the first step is a point. If at each step the circle bundle
is principal, the result is an {\em iterated principal circle bundle.} 

This note was prompted by a question of Xiaochun Rong who
asked me to justify the following fact mentioned in~\cite{BelWei-gafa}:

\begin{thm}
\label{thm: nil}
A manifold is an iterated principal circle bundle if and only if 
it is a nilmanifold.
\end{thm}

The proof of Theorem~\ref{thm: nil} combines some bundle-theoretic considerations
with classical results of Mal'cev~\cite{Mal}. 
The ``if'' direction was surely known since~\cite{Mal}
but~\cite[Proposition 3.1]{FadHus} seems to be the earliest reference. 
The statement of Theorem~\ref{thm: nil} is mentioned 
without proof in~\cite[p.98]{Weinb-book} and \cite[p.122]{FOT}. 

{\bf Summary of previous work:}

(1) 
Every iterated principal circle bundle has torsion-free nilpotent 
fundamental group because the homotopy exact sequence converts
a principal circle bundle into a central extension with infinite cyclic kernel.

(2) 
Theorem 1.2 of~\cite{Nak} implies that  every iterated principal circle bundles 
is diffeomorphic to an infranilmanifold; this was explained to me by Xiaochun Rong. 
Thus~\cite{Nak} gives another (less elementary) proof 
of the ``only if'' direction in Theorem~\ref{thm: nil} because
every iterated principal circle bundle is homotopy equivalent to a nilmanifold,
and the diffeomorphism type of an infranilmanifold is determined by its homotopy type~\cite{LeeRay}.

(3)
According to~\cite{PalStu} a manifold is a principal torus bundle over a torus if and only if 
it is a nilmanifold modelled on a two-step nilpotent Lie group.

(4) Every $3$-dimensional infranilmanifold has a unique Seifert fiber space 
structure, see~\cite[Theorem 3.8]{Sco}, hence it is an iterated circle bundle if and only if the base
orbifold (of the Seifert fibering) is non-singular, i.e., the $2$-torus or the Klein bottle.
Thus iterated circle bundles are rare among $3$-dimensional infranilmanifolds.

(5) In~\cite{LeeMas} it is proven that every iterated circle bundle is homeomorphic 
to an infranilmanifold. Their argument splits in two parts: finding a homotopy equivalence 
and upgrading it to a homeomorphism. The latter uses topological surgery,
which does not extend to the smooth setting. 

(6) A natural way to establish the smooth version of the above-mentioned result in~\cite{LeeMas}
is to show that every iterated circle bundle is almost flat, and then
apply the celebrated work of Gromov-Ruh~\cite{Gro, Ruh} that infranilmanifolds
are precisely the almost flat manifolds. Recall that a closed manifold is
{\em almost flat\,} if it admits a sequence of Riemannian metrics 
of uniformly bounded diameters and sectional 
curvatures approaching zero. To this end we prove:

\begin{thm}
\label{thm: infra}
Any iterated circle bundle is almost flat, and therefore diffeomorphic to
an infranilmanifold.
\end{thm}

\begin{proof}[Proof of Theorem~\textup{\ref{thm: nil}}]
We use~\cite[Chapter II]{Rag-book} as a reference for Mal'cev's work.
If $N/\Gamma$ is a nilmanifold, then $\Gamma$ 
is finitely generated, torsion-free, and nilpotent, and conversely, any such group
is the fundamental group of a nilmanifold, see~\cite[Theorem 2.18]{Rag-book}.
Every automorphism of $\Gamma$ extends uniquely to an automorphism of 
$N$, see~\cite[Theorem 2.11]{Rag-book}.
Applying this to conjugation by an element of the center of $\Gamma$ we get
the inclusion of centers $Z(\Gamma)\subset Z(N)$. 
Nilpotency of $\G$ ensures that $Z(\Gamma)$ is nontrivial, and therefore,
there is a one-parameter subgroup $R\le Z(N)$ such that $R\cap Z(\Gamma)$ is nontrivial, and hence
infinite cyclic. Clearly $R\cap \Gamma=R\cap Z(\Gamma)$. 
The left $R$-action on $N$
descends to a free $R/(R\cap \Gamma)$-action on $N/\Gamma$, which makes
$N/\Gamma$ into a principal circle bundle whose base $B_{_\Gamma}$ is a nilmanifold, namely,
the quotient of $N/R$ by $\Gamma/(R\cap\Gamma)$.
This proves the ``if'' direction.

Conversely, let $p\co E\to B$ be a principal circle bundle over a nilmanifold $B$. 
Its homotopy exact sequence is a central extension, 
so $\pi_1(E)$ is finitely generated torsion-free nilpotent. Consider
a nilmanifold $N/\Gamma$ with $\Gamma\cong\pi_1(E)$, and
let $z\in Z(\Gamma)$ be the element corresponding to the circle fiber of $p$
through the basepoint.
Let $R\le N$ be the one-parameter subgroup that contains $z$. 
As above $R\subset Z(N)$ and
$N/\Gamma$ is the total space of a principal circle bundle $p_{_\Gamma}\co N/\Gamma\to B_{_\Gamma}$
whose base $B_{_\Gamma}$ is a nilmanifold and the
fibers are the $R/(R\cap \Gamma)$-orbits. 
The cyclic group $R\cap \Gamma$ is generated by $z$ 
because its generator projects to a finite order element 
in the torsion-free group $\G/\langle z\rangle\cong\pi_1(B)$.
Thus the isomorphism $\pi_1(E)\cong\pi_1(N/\Gamma)$ descends to
an isomorphism $\pi_1(B)\to \pi_1(B_{_\Gamma})$.
Since all these manifolds are aspherical, the fundamental group isomorphisms
are induced by homotopy equivalences, and we get a homotopy-commutative square
\begin{equation*}
\xymatrix{
E\ar[d]_p\ar[r]^{\!\e}&
N/\Gamma\ar[d]^{p_{_{\Gamma}}}\\
B\ar[r]_{\!\b} & B_{_\Gamma}
} 
\end{equation*}
where $\e$ and $\b$ are homotopy equivalences.
We can assume that $\b$ is a diffeomorphism because by~\cite[Theorem 2.11]{Rag-book} 
any homotopy equivalence of nilmanifolds is homotopic to a 
diffeomorphism. 
The Gysin sequence implies that the Euler class of a circle bundle generates 
the kernel of the homomorphism induced on the second cohomology
by the bundle projection.
The map of the Gysin sequences of $p$ and $p_{_{\Gamma}}$ induced by
the commutative square shows that
$\b$ preserves their Euler classes up to sign, and after changing
the orientation if necessary we can assume that the Euler classes are preserved by $\b$.
The isomorphism type of a principal circle bundle 
is determined by its Euler class.
Since $p$ and the pullback of $p_{_\Gamma}$ via $\beta$ have the same Euler class, they 
are isomorphic, which gives a desired diffeomorphism of $E$ and $N/\Gamma$
and completes the proof of the ``only if'' direction. 
\end{proof}

\begin{proof}[Proof of Theorem~\ref{thm: infra}] 
In view of~\cite{Gro, Ruh} 
it is enough to prove inductively that the total space of 
any circle bundle over an almost flat manifold 
is almost flat. This comes via the following standard argument.
Let $p\co E\to B$ be a smooth circle bundle over a closed manifold $B$.
For any Riemannian metric $\check g$ on $B$ there is a metric 
$g$ on $E$ such that $p$ is a Riemannian submersion with totally geodesic fibers
which are isometric to the unit circle, see~\cite[9.59]{Bes-book}. 
As in~\cite[9.67]{Bes-book}
let $g^t$ be the metric on $E$ obtained by rescaling $g$ by a positive constant $t$ along the fibers of $p$,
i.e., $g^t$ and $g$ have the same vertical and horizontal distributions $\mathcal V$, $\mathcal H$, and
$g^t\vert_{\mathcal V}=t g\vert_{\mathcal V}$ and  $g^t\vert_{\mathcal H}=g\vert_{\mathcal H}$.
The fibers of $p$ are $g^t$-totally geodesic~\cite[9.68]{Bes-book} so the $T$ tensor vanishes. 
The diameters of $g^t$, $\check g$ satisfy $\mathrm{diam}(g^t)\le\diam(\check g)+O(\sqrt{t})$.
The following lemma finishes the proof of almost flatness of $E$.
\end{proof}

\begin{lem}
The sectional curvatures $K^t$, $\check K$ of $g^t$, $\check g$ satisfy
$|K^t|\le |\check K|+ O(\sqrt{t})$.
\end{lem}
\begin{proof}
Fix any $2$-plane $\s$ tangent to $E$. 
Since $\mathcal H$ has codimension one, $\s$ contains 
a $g^t$-unit horizontal vector $X$. Let $C$ be a $g^t$-unit vector in $\s$ that is 
$g^t$-orthogonal to $X$. Write $C=U+Y$ where $U\in\mathcal V$, $Y\in\mathcal H$. 
The sectional curvature of $\s$ with respect to  $g^t$ is given by 
\[
\vspace{2pt}
K^t_\s=\langle R^t(C,X)C, X\rangle^t =
\langle R^t(Y,X)Y, X\rangle^t +
2\langle R^t(Y,X)U, X\rangle^t +
\langle R^t(U,X) U, X\rangle^t
\vspace{2pt}
\]
where $\langle C,D\rangle^t:=g^t(C,D)$ and $R^t$ is the curvature tensor of $g^t$. 

Lemma 9.69 of~\cite{Bes-book} relates the $A$ tensors $A^t$, $A$ of $g^t$, $g$
as follows: $A^t_Y X=A_Y X$ and $A^t_X U=t\,A_X U$.
Recall that $A_Y X$ is vertical and $A_X U$ is horizontal. 
The formulas in~\cite[9.28, 9.69]{Bes-book} give\vspace{5pt}
\[
\check g(\check R(\check Y,\check X)\check Y, \check X)-\langle R^t(Y,X) Y, X\rangle^t=
3\langle A^t_Y X, A^t_Y X\rangle^t=3t\,g(A_Y X, A_Y X)
\]
\vspace{-2pt}
\[
\langle R^t(Y,X)U, X\rangle^t=-[\langle (D_X A)_Y X, U\rangle ]^t=-t\,g((D_X A)_Y X, U) 
\]
\vspace{-2pt}
\[
\langle R^t(U,X)U, X\rangle^t=\langle A^t_X U, A^t_X U\rangle^t+ [\langle (D_U A)_X X, U\rangle ]^t
=t^2 g(A_X U, A_X U)\vspace{7pt}
\]
where $[\langle (D_U A)_X X, U\rangle ]^t=0$ by the last formula in~\cite[9.32]{Bes-book}.

Since $g(X,X)=1=g^t(C,C)=g(Y,Y)+tg(U,U)$, 
the vectors $X$, $Y$, $\sqrt{t}\,U$ lie in the $g$-unit disk bundle of $TE$, which is compact, so 
the functions $g(A_Y X, A_Y X)$, $\sqrt{t}\,g((D_X A)_Y X, U)$, $t\, g(A_X U, A_X U)$ are bounded. 

Therefore, if $Y\neq 0$ and $\check\s$ is the projection of $\s$ in $TB$, 
then \[
K^t_\s=\check g(\check R(\check Y,\check X)\check Y, \check X)+
O(\sqrt{t})
=\sqrt{{\check g}(\check Y,\check Y)}\, K_{\check\s}+O(\sqrt{t})\]
and if $Y=0$, then $K^t_\s=t^2 g(A_X U, A_X U)=O(t)$. 
Thus $|K^t_\s|\le |K_{\check\s}|+O(\sqrt{t})$. 
\end{proof}
\vspace{3pt}

\bibliographystyle{amsalpha}
\bibliography{nil-S1}
\end{document}